\documentclass{amsart}

\usepackage{amsmath,amsfonts,amssymb}
\newtheorem{teo}{Theorem}
\newtheorem{rk}{Remark}

\newtheorem{clly}{Corollary}
\newtheorem{lemma}{Lemma}

\newcommand{\R}{{\mathbb{R}}}

\newcommand{\Z}{{\mathbb{Z}}}
\newcommand{\N}{{\mathbb{N}}}

\newcommand{\fix}{Fix}
\newcommand{\ta}{TA}

\newcommand{\inte}{Int}

\newcommand{\homeo}{Homeo}
\begin{document}
\title{Topologically Anosov plane homeomorphisms.}
\author{Gonzalo Cousillas, Jorge Groisman and Juliana Xavier}

\address{ Instituto de Matemática y Estadística ``Rafael Laguardia'', Facultad de Ingenier\'{\i}a,  Universidad de la Rep\'ublica, Montevideo, Uruguay.}
\email{gcousillas@fing.edu.uy}
\email{jorgeg@fing.edu..uy}
\email{jxavier@fing.edu.uy}


 \begin{abstract}  This paper deals with classifying the dynamics of {\it Topologically Anosov} plane homeomorphisms.  We prove that a Topologically Anosov homeomorphism $f:\R^2 \to \R^2$
 is conjugate to a homothety if it is  the time one map of a flow. We also obtain results for the cases when the nonwandering set of $f$  reduces to a fixed point, or if  there exists an open, connected, simply connected  proper subset $U$ such that $U \subset \inte(\overline {f(U)})$, and such that $ \cup_{n\geq 0} f^n (U)= \R ^2$.  In the general case, we prove a structure theorem for the $\alpha$-limits of orbits with empty $\omega$-limit (or the $\omega$-limits of orbits with 
empty $\alpha$-limit), and we show that any basin of attraction (or repulsion) must be unbounded.

 \end{abstract}
 \maketitle

 \section{Introduction}
 A homeomorphism $f:M\to M$ of the 
 metric space to itself is called {\it expansive} if there exists $\alpha >0$ such that  given $x,y \in M, x\neq y$, then
 $d(f^n (x), f^n (y) )> \alpha$ for some $n\in \Z$.  The number $\alpha$ is called the {\it expansivity constant} of $f$. 
 
 The study of expansive systems is both classic and fascinating.  In Lewowicz's words \cite{lew}, the fact that every point has a distinctive dynamical meaning implies that a rich interaction
 between dynamics and topology is to be expected. 
 
If $\delta >0$, a $\delta$-pseudo-orbit for $f$ is a sequence $(x_n)_{n\in \Z}$ such that $d(f(x_n), x_{n+1})< \delta$ for all $n\in \Z$. If $\epsilon >0$, we say that the orbit of $x$ $\epsilon$-shadows a given pseudo-orbit if 
$d(x_n, f^n (x))<\epsilon$ for all $n\in \Z$. Finally,  we  say  that $f$
has  the  shadowing  property  if for each $\epsilon >0$ there exists $\delta >0$ such that every $\delta$-pseudo-orbit
is $\epsilon$-shadowed  by  an  orbit  of $f$. In other words, systems with the
shadowing property are precisely the ones in which ``observational errors'' do not
introduce unexpected behavior, in the sense that simulated orbits actually  ``follow''
real orbits.

 Anosov diffeomorphisms, the best known chaotic dynamical systems, are expansive and have the shadowing property.  Moreover, expansive homeomorphisms with 
 the shadowing property on compact metric spaces are known to have spectral decomposition in Smale's sense (\cite{aoki}).
 
 On non-compact spaces however, it is well known that a dynamical system may be expansive or have the shadowing property with respect to one metric, but not with respect to another metric that 
 induces the same topology. In \cite{dlrw} topological definitions of expansiveness and shadowing are given that are equivalent to the usual metric definitions for homeomorphisms on compact
 metric spaces, but are independent of any change of compatible metric.  In \cite{cou}, the author applies these definitions with the plane $\R ^2$ as the phase space and proves a fixed
 point theorem. Following his spirit, we take these definitions and try to classify the dynamics with the plane $\R ^2$ as the phase space.

Let $f: \R^2 \to \R ^2$ be a continuous map and $\delta: \R^2\to \R $ a continuous and strictly
  positive function.  A $\delta$-{\it pseudo-orbit} for $f$ is a sequence $(x_n)_{n\in \N}\subset \R^2$ such that $||f(x_n) - x_{n+1}||< \delta (f(x_n))$.
  A $\delta$-pseudo-orbit $(x_n)_{n\in \N}$ is {\it $\epsilon$-shadowed} by an orbit, if there exists $x\in \R^2$ such that $||x_n - f^n (x)||< \epsilon (x_n)$ for all $n\in \Z$.

 Throughout this paper $f:\R^2\to \R^2$ is a {\it Topologically Anosov} ($\ta$) homeomorphism.  That is:
 \begin{itemize}
  \item it is {topologically expansive}: there exists a continuous and strictly positive function $\epsilon: \R^2\to \R$ such that for all $x, y \in \R^2, x\neq y$ there exists $k\in \Z$ 
  satisfying $||f^k (x)-f^k (y)|| > \epsilon (f^k (x))$;
  \item it satisfies the {\it topological shadowing property}: for all continuous and strictly positive function $\epsilon: \R^2\to \R$ there exists $\delta: \R^2\to \R $ a continuous and strictly
  positive function such that every $\delta$-pseudo-orbit is $\epsilon$-shadowed by an orbit. 
 \end{itemize}
 
 As an example, a rigid translation is topologically expansive but does not satisfy the topological shadowing property. An example of $\ta$ homeomorphism is any homothety
 (see \cite{cou} for a proof), following the same ideas it can be seen that a {\it reverse homothety} (by a reverse homothety we mean  the map $z\mapsto 2\bar{z},\;z\in\mathbb{C}$)  is also a $\ta$ homeomorphism.  As being $\ta$ is a conjugacy invariant, the whole conjugacy class of homotheties and reverse homotheties belongs to the family of $\ta$ homeomorphisms.  In this work we deal
 with the problem of classifying $\ta$ homeomorphisms.  In particular, are all $\ta$ homeomorphisms conjugate to a homothety or  a reverse homothety?  We prove that this is the case if the homeomorphism is the time
 one map of a flow defined by a $C^1$ vector field (Theorem \ref{teo1}). If there is a global attracting fixed point $x_0$ (that is, $f^n (x)\to x_0$ for all $x\in \R^2$), we prove 
  that $f$ must be  also conjugate to a homothety or a reverse homothety.  What about an expansive attractor?  Is the Plykin attractor $\ta$?  We prove it is not, at least if its basin of attraction is the whole plane.  More generally, 
 we prove that if there exists an open, connected,  simply connected  proper subset $U$ such that $\overline {f(U)} \subset \inte (U)$, and such that $ \cup_{n\leq 0} f^n (U)= \R ^2$, then  $K = \cap_{n\geq 0} f^n (U) $ must be a single point.  Finally, we prove that if $f\in \homeo(\R ^ 2)$ is $\ta$, and $\Omega(f) = \{x_0\}$, $x_0 \in \fix (f)$, then $f$ is conjugate to a homothety if $f$ is  orientation preserving , and conjugate to a reverse homothety if $f$ is orientation reversing.

 \section{The one-dimensional case}
 
 In this brief section we characterize Topologically Anosov homeomorphisms on $\R$. 

\begin{teo}\label{DimOneCase}
Let $f$ be a Topologically Anosov homeomorphism on $\R$. Then, $f$ is topologically conjugate to $g$ where $g(x)=\pm \frac{x-x_{0}}{2} +x_{0}$, depending on whether $f$ preserves or reverses 
orientation.
\end{teo}

Let us prove some useful lemmas.

\begin{lemma}  
Let $f:\R \to \R$ be a Topologically Anosov homeomorphism. Then, there exists a unique fixed point for $f$.
\end{lemma}

\begin{proof}
If $f$ reverses orientation, it is clear that $\fix(f)=\{p\}$, for some $p\in \R$. If $f$ is orientation preserving, and fixed point free, then $f$ is topologically conjugate to a translation,
which does not admit the shadowing property.

Regarding uniqueness, suppose that $x_{1}<x_{2}$ are fixed points and let $g$ be the restriction of $f$ to $[x_{1},x_{2}]$. Then $g$ is
a metric expansive homeomorphism in a compact interval, contradicting Bryant's theorem in \cite{bry}. 
\end{proof}

\begin{proof} {\it of Theorem \ref{DimOneCase}.}  

By the previous lemma, $\fix (f) = \{x_0\}$. Without loss of generality, let us consider the case $x_{0}=0$.  We deal first with the orientation preserving case.

Consider $h:\R \rightarrow \R$ defined as follow:
\begin{itemize}
\item $h(0)=0$.
\item Fix some point $p\in \R^{+}$ and define $h(p)=q$ where $q$ is an arbitrarily point of $\R^{+}$. Let $g_1:\R \rightarrow \R$ and  $g_2:\R \rightarrow \R$ defined as $g_1(x)=2x$ and $g_2(x)= x/2$. Then, if $f^{n}(p)$ tends to $\infty$ define $h(f^{n}(p))=g_1^{n}(h(p))$, $n\in \Z$ and if $f^{n}(p)$ tends to $0$ define $h(f^{n}(p))=g_2^{n}(h(p))$, $n\in \Z$.
\item In the open interval $(p,f(p))$ define $h$ as an arbitrary increasing homeomorphism.
\item Finally, let $x>0$ an arbitrary point. Then, $x\in [f^{k}(p),f^{k+1}(p)]$ for some $k\in \Z$. Thus we define $h$ in $x$ as $h(x)=g_i^{k}(h(f^{-k}(x))),\ i=1,2$, depending on whether $f^{n}(p)$ tends to $\infty$ or to $0$.
\item The construction is the same for $x\in \R^{-}$.
\item We claim that $0$ is a global repeller or attractor and then conjugate to $g_i, \ i=1,2$ respectively. If not there exists, 
$q<0$ and $p>0$ such that $d(q,p)<\delta$, $f^{n}(p)$ tends to $\infty, \ n\rightarrow +\infty$ and  $f^{n}(q)$ tends to $0, \ n\rightarrow +\infty$ (or viceversa). So, given an arbitrary $\delta>0$ consider a $\delta$-pseudo orbit $(x_n)$ defined as: $x_n= f^{n}(q)$ for $n\leq 0$, and $x_n= f^{n-1}(p)$ for $n\geq 1$. It is clear that there is not orbit that $\epsilon$-shadows  $(x_n)$ for a convenient $\epsilon$. This proves the claim.  

\end{itemize}

If $f$ reverses orientation, we know that $f^2$ is an orientation preserving Topologically Anosov homeomorphism and then conjugate to a homothety. We also have that $\fix (f^2)=\{0\}$ (if not we have a contradiction with the expansivity of $f^2$). Thus, every point $p\in \R$ verifies that $f^{2n}(p)$ tends monotonously to $\infty$ or to $0$ when $n$ tends to $+\infty$. But this implies that $f^{2n+1}(p)$ tends monotonously to $\infty$ or to $0$ when $n$ tends to $+\infty$. So, we are able to define a conjugation between $f$ and $g_1(x)=-2x$ if $0$ is a repeller ($g_2(x)= -x/2$ if $0$ is an attractor) in the same way we did in the orientation preserving  case.  
\end{proof}

 \section{Non accumulating future (or past) orbits}
 
Points with empty $\alpha$ or $\omega$-limits are specially important for the study of $\ta$ plane homeomorphisms. We explain why in this section.

\begin{lemma}\label{epsilon}
Let $f\in \homeo (\R ^2)$. If $\omega(x) = \emptyset$ there exists $\epsilon: \R^ 2 \to \R$ a continuous positive map with the property that if $y\neq x$, then there exists 
$ n>0$ such that 
$||f^n (x)- f^n (y)|| > \epsilon (f^n (x))$.  In particular, if $(x_n)_{n\in \Z}$ is a pseudo-orbit such that $x_n = f^n (x)$ for all $n\geq n_0$, then the only possible orbit that
$\epsilon$-shadows $(x_n)_{n\in \Z}$ is that of $x$.
 
\end{lemma}

\begin{proof}  First note $\omega(x) = \emptyset$ implies that there exists a family of pairwise disjoint open sets $(U_n)_{n\in \N}$ such that each $U_n$ is a neighborhood of $f^n (x)$.
We claim that there exists a family of open sets $(V_n)_{n\in \N}$ such that for all $n\in \N$, $V_n \subset U_n$, $V_n$ is a neighborhood of $f^n (x)$, and a continuous map 
$h:\cup_n V_n\to \R^2$ which is a homeomorphism onto its image such that $hf|_{\cup_n V_n}=Th$, where $T(x,y) = (x+1, y)$ for all $(x,y)\in \R ^2$.
Take a homeomorphism $h: U_0 \to B((0,0), 1/3)$, and let $V_0\subset U_0 $ be an open set containing $x$ such that $f(V_0)\subset U_1$.  Define $\tilde U_1 := f(V_0)$ and extend the homeomorphism
$h$ to $\tilde U_1$ as $h|_{\tilde U_1} = Thf^{-1}$.  Note that $hf|_ {V_0} = Th|_{V_0}$. We now define $V_1\subset \tilde U_1$ such that $f(V_1)\subset U_2$, let $\tilde U_2 = f (V_1)$
and extend $h$ to $\tilde U_2$ as $h|_{\tilde U_2} = Thf^{-1}$.  Inductively, if $h$ is defined on $\tilde U_i\subset U_i$, we extend $h$ to $\tilde U_{i+1}\subset U_{i+1}$ as follows.
We take $V_i \subset \tilde U_i$ such that $f(V_i) \subset U_{i+1}$ and let $\tilde U_{i+1} = f(V_i)$.  We then let $h|_{\tilde U_{i+1}} = Th f^{-1}$.  Note that for all $i$, 
$hf|_ {V_i} = Th|_{V_i}$.  This proves the claim. 

Now take $\tilde \epsilon: \R^ 2 \to \R$ a continuous positive map verifying that for all $n\in \N$, $B((k,0), \tilde \epsilon ((k,0)))\subset h (V_k)$ and also that if $y\neq x$, then there exists 
$ n>0$ such that 
$||T^n (x)- T^n (y)|| > \tilde \epsilon (T^n (x))$.  Finally, we define $\epsilon: \R^ 2 \to \R$ such that $B(f^n (x),\epsilon (f^n (x)))\subset h^{-1}(B((k,0)),\tilde \epsilon ((k,0)))$ 
and extend it
to a continuous positive map of $\R ^2$.  To check that this map satisfies the condition of the lemma, just notice that if for some $y$, $f^n (y)\in V_n$ for all $n\in N$, then 
$T^n h(y)= h (f^n (y))$ and $T$ does not satisfy the topological shadowing property..
\end{proof}

\begin{lemma}\label{inftyinfty}  Let $f\in \homeo (\R ^2)$ be a $\ta$ homeomorphism. If $\alpha (x) = \emptyset$, then $\omega (x)\neq \emptyset$.
 
\end{lemma}

\begin{proof} 

 If $\alpha (x) = \omega(x) = \infty$,  by Lemma \ref{epsilon} there exists 
 $\epsilon: \R^ 2 \to \R$ a continuous positive map with the property that if $y\neq x$, then there exists
$ n>0$ such that 
$||f^n (x)- f^n (y)|| > \epsilon (f^n (x))$ and $||f^{-n} (x)- f^{-n} (y)|| > \epsilon (f^{-n} (x))$. Take $\delta: \R^ 2 \to \R$ a continuous positive map as in the 
definition of shadowing, and consider the following $\delta$-pseudo-orbit $(x_n)_{n\in \Z}$:  $x_n = f^{n} (x)$ for all $n<0$ ; $x_n= f^n (y)$ for all $n\geq 0$, where 
$y\in B (x, \delta (x))$.  Then, the orbit of $x$ must $\epsilon$-shadow this pseudo-orbit, but this is impossible by the choice of the map $\epsilon$.
\end{proof}

For the remainder of  this section $f\in \homeo (\R ^2)$ is assumed to be $\ta$ and $z_0\in \fix(f)$.

\begin{lemma}\label{stable} If there exists $z\in \R^2$ such that $\alpha (z) = \emptyset$ and $\omega(z) = \{z_0\}$, 
then $z_0$ is Lyapunov stable. 
 
\end{lemma}

\begin{proof}  By Lemma \ref{epsilon} there exists $\mathcal{E}: \R^ 2 \to \R$ a continuous positive map with the property that
 if $(x_n)_{n\in \Z}$ is a pseudo-orbit such that $x_n = f^n (z)$ for all $n\leq n_0$, then the only possible orbit that $\mathcal{E}$-shadows $(x_n)_{n\in\Z}$ is that of $z$ because $\alpha(z)=\emptyset$.   Given $\epsilon >0$, take $n_0$ such that $f^n (z)\notin B(z_0, \epsilon)$ for all $n\leq n_0$, and construct $\xi: \R^ 2 \to \R$ a continuous 
positive map such that $\xi (z_0) = \epsilon$ and $\xi (f^n (z)) = 
\mathcal{E}(f^n (z))$ for all $n\leq n_0$.  Take $\delta: \R^ 2 \to \R$ a continuous positive map,
such that every $\delta$-pseudo-orbit is $\xi$-shadowed by an orbit.  It follows that $y\in B(z_0, \delta (z_0))$ implies $f^n (y)\in B (z_0, \epsilon)$ for all $n\geq 0$ 
(otherwise there exists a $\delta$-pseudo-orbit
that cannot be $\xi$-shadowed).
\end{proof}

\begin{lemma}\label{x0x0}  If there exists $x\neq z_0$ such that $\alpha (x) = \omega(x) = \{z_0\}$, then there exists $y_0\neq z_0$ and $z$ such that $y_0\in \omega(z)$.
 
\end{lemma}

\begin{proof} Suppose that  $\alpha (x) = \omega(x) = \{z_0\}$ and take $\epsilon: \R^ 2 \to \R$ a continuous 
positive map such that the entire orbit of $x$ is not contained in $B_0=B(z_0, \epsilon (z_0))$.  Modifying the function $\epsilon$ if necessary, we may assume that
$B(f^n (x), \epsilon(f^n (x)))\cap B_0 = \emptyset$ for all $n$ such that $f^n (x)\notin B_0$.  Take $\delta: \R^ 2 \to \R$ a continuous positive map as in the 
definition of shadowing, and take positive integers $N,M$ big enough such that $d(f^{-N} (x), f^M (x))<\min \{\delta(z): z\in \overline{B_0}\}$. Then, $w_0 = z_0 , w_{i+1} = f^{-N+i}(x), 
i= 0, \ldots, M+N -1$, $w_M = z_0$ defines a periodic $\delta$-pseudo-orbit $(w_n)_{n\in \Z}$.
 Note that if an orbit $z$, $\epsilon$-shadows this pseudo-orbit, it must visit
infinitely many times any $B(f^n (x), \epsilon(f^n (x)))$  such that $f^n (x)\notin B_0$.  Therefore, there exists $y_0\neq z_0$ such that $y_0\in \omega(z)$.
\end{proof}

\begin{lemma}\label{infi} If $\Omega(f) = \{z_0\}$,  then there exists $x\in \R^2$ such that $\alpha(x) = \emptyset$ or $\omega(x) = \emptyset$.
 
\end{lemma}

\begin{proof}  First note that as $\Omega(f) = \{z_0\}$, for all $x$ the sets $\alpha (x)$ and $\omega(x)$ are either empty or the single point $z_0$ (as  
$y \in \alpha(x)\cup \omega(x)$ implies $y \in \Omega(f)$).

We finish the proof by pointing out  that if $\alpha (x) = \{z_0\}$, then $\omega (x) = \emptyset$ because of the preceeding lemma. 
\end{proof}

\begin{lemma}  If $\Omega (f) \neq \{z_0\}$, then there exists $y_0\neq z_0$ and $z$ such that $y_0\in \omega(z)$.
 
\end{lemma}

\begin{proof}  Take $x\neq z_0\in \Omega (f)$ and note that we may assume that $x\notin \fix (f)$ (otherwise we are done with the proof).  Take $\alpha >0$ such that $B(z_0, \alpha)$, $B(x, \alpha)$ and $B (f(x), \alpha)$ are 
pairwise disjoint.  Take $\epsilon: \R^ 2 \to \R$ a continuous 
positive map such that $\epsilon (z_0) = \epsilon (x) = \epsilon (f(x))= \alpha $ and  take $\delta: \R^ 2 \to \R$ a continuous positive map as in the 
definition of shadowing. Take $0<\beta < \delta(x)/2$ such that $f(B(x,\beta))\subset B(f(x), \delta (f(x))/2)$.  As $x\in\Omega(f)$, there exists $y$ and $n>0$ such that both $y$ and 
$f^n (y)$ belong to $B(x, \beta)$.  Then, $f(y)$  belongs to $B(f(x), \delta(f(x))/2)$.  Then construct the following periodic $\delta$-pseudo orbit:  $x_0 = x$, 
$x_i= f^i (y)$ for all $i= 1, \ldots, n-1$, $x_n = x = x_0$.  This pseudo-orbit must be $\epsilon$-shadowed by an orbit $z$.  Therefore, the orbit of $z$ must visit infinitely many times 
$B(x, \alpha)$, and the result follows.
\end{proof}


We obtain our first result: 
 
\begin{teo}  If $\Omega(f) = \{z_0\}$, then $f$ is conjugate to a homothety or  a reverse homothety. 
 
\end{teo}

\begin{proof}  As was already pointed out,   for all $x$ the sets $\alpha (x)$ and $\omega(x)$ are either empty or the single point $z_0$ (as  
$y \in \alpha(x)\cup \omega(x)$ implies $y \in \Omega(f)$).

By Lemma \ref{infi}  there exists $x\in \R^2$ such that $\alpha(x) = \emptyset$ or $\omega(x) = \emptyset$ 

Moreover, if $\alpha (x) = \emptyset$, then $\omega (x) = \{z_0\}$ (indeed, Lemma \ref{inftyinfty} implies that $\omega(x)\neq \emptyset$).

Finally, we claim that if there exists $x$ such that $\alpha (x) = \emptyset$ (and therefore $\omega (x) = \{z_0\}$), then every $z\neq z_0$ verifies $\alpha (z) = \emptyset$ (and therefore
$\omega (z) = \{z_0\}$).  Indeed, by Lemma \ref{stable}, $x_0$ is Lyapunov stable, which implies that any $z\neq z_0$ such that $\alpha (z) =\{z_0\}$ must verify also $\omega (z) =\{z_0\}$,
which is impossible by Lemma \ref{x0x0}.  We conclude that if there exists $x$ such that $\alpha (x) = \emptyset$, then $z_0$ is a {\it global attractor}, that is, 
$\lim_{n\to +\infty} f^n (z) = z_0$ for all $z\in \R^2$.

If there is no $x$ such that $\alpha (x) = \emptyset$, then $\alpha(x) = \{z_0\}$ for all $x$, and therefore  $\omega (x) = \emptyset$ for all $x$.

We have proven that  $z_0$ is either a global attractor or a global repeller which is asymptotically stable. The result now follow from Ker\'kj\'art\'o's  theorem 
(\cite{k1}, \cite{k2}, or for a more modern approach \cite{g}).
 \end{proof}


\begin{clly}   If $z_0$ satisfies $\lim _{n\to +\infty} f^n (z) = z_0$ for all $z\in \R ^2$, then $f$ is conjugate to a homothety or a reverse homothety.
 
\end{clly}

\begin{proof}  In this case, $\Omega(f)= \{z_0\}$ and we are done by the previous theorem.  
\end{proof}

 We finish this section by describing the possible $\omega$- (or $\alpha$)-limits of points with non accumulating past (or future) orbits.

\begin{lemma}\label{una}  If  $\omega(x)= \emptyset$, then $\alpha (x)$ contains at most one periodic orbit. 
 
\end{lemma}

\begin{proof} By Lemma \ref{epsilon}, there exists $\epsilon: \R^ 2 \to \R$ a continuous positive map with the property that if $y\neq x$, then there exists 
$n\in \Z, n>0$ such that 
$||f^n (x)- f^n (y)|| > \epsilon (f^n (x))$.  In particular, if $(x_n)_{n\in \Z}$ is a pseudo-orbit such that $x_n = f^n (x)$ for all $n\geq n_0$, then the only possible orbit that
$\epsilon$-shadows $(x_n)_n$ is that of $x$.  Suppose that  $\alpha (x)$ contains two different periodic orbits $z_1$ and $z_2$.  Modifying the function $\epsilon$ if necessary, we may
assume that 
$B (f^n (z_i), \epsilon (f^n (z_i))) \cap B (f^m (z_j), \epsilon (f^m (z_j)))\neq \emptyset$ if and only if $i=j$ and $m= n$.
 Take $\delta: \R^ 2 \to \R$ a continuous positive map,
such that every $\delta$-pseudo-orbit is $\epsilon$-shadowed by an orbit and take $n_1, n_2 $ positive integers such that $f^{-n_1} (x) \in B(z_1, \delta (z_1))$ and 
 $f^{-n_2} (x) \in B(z_2, \delta (z_2))$.  Construct now two $\delta$-pseudo orbits $(x^1_n)_{n\in \Z}$ and  $(x^2_n)_{n\in \Z}$ as follows:  $x^1_n =f^{(-n_1-n)} (z_1)$ for all $n< -n_1$; 
 $x^1_n = f^n (x)$ for all $n\geq -n_1$;   $x^2_n =f^{(-n_2-n)} (z_2)$ for all $n< -n_2$; 
 $x^2_n = f^n (x)$ for all $n\geq -n_2$.  As noted above,  then the only possible orbit that $\epsilon$-shadows any of these pseudo-orbits is that of $x$.  However, if the orbit of
 $x$ $\epsilon$-shadows the pseudo-orbit $(x^1 _n)_n$, $f^n (x)\in B(x^1_n,\epsilon (x^1_n) )$ for all $n<-n_1$.
This clearly implies that the orbit of $x$ cannot $\epsilon$-shadow the pseudo-orbit $(x^2 _n)_n$, a contradiction. 
\end{proof}

 We recall the classic Utz's result that will be used in the next lemma:
 
 \begin{teo}\label{utz}  If $K$ is compact and supports a future-expansive homeomorphism, then it is finite.
  
 \end{teo}

 Recall that a map is future expansive if there exists $\epsilon >0$ such that $x\neq y$ implies there exists $n\geq 0$ such that $d(f^n (x), f^n (y)) > \epsilon$.
 
 \begin{lemma}\label{ctes}  Let $K$ be compact and invariant, and suppose there exists $\alpha>0$, $C>0$, such that $d(x,y)<\alpha$ implies that there exists $j>0$ such that $d(f^j (x), f^j (y))>C$.
 Then, $K$ is finite.
  
 \end{lemma}

 \begin{proof} Note that if $C> \alpha$, we get that $f|_K$ is $\alpha$-future expansive. If $C\leq \alpha$ we get that $f|_K$ is $C$-future expansive.  In any case,  $K$ must be finite
 by Theorem \ref{utz}.
 \end{proof}

\begin{lemma}\label{autz}  Let $K$ be a compact invariant set with expansivity constant $C$.  Suppose that for all $x\in K$ there exists a neighborhood $U$ of $x$, and $z\in U$ such that the orbit of $z$
$C/2$-shadows
any pseudo-orbit $(x_n)_{n\in \Z}$ such that $x_n = f^n (y), n<0$ for some $ y \in U$ and $x_n = f (z), n\geq 0$.  Then, $K$ is finite.
 
\end{lemma}

\begin{proof} Take  a finite cover of $K$ with neighborhoods as in the hypothesis of the lemma. Let $\alpha >0$ be such that $d(x,y)<\alpha$, then $x$ and $y$ belong to one of the balls of
such cover. Then, if $d(x,y)<\alpha$, both pseudo-orbits $x_n = f^n (x), n<0$, and $x_n = f (z), n\geq 0$ and 
$y_n = f^n (y), n<0$, and $y_n = f (z), n\geq 0$ are $C/2$-shadowed by the orbit of $z$, and therefore  $d(f^{-n}(x), f^{-n}(y)) < C$ for all $n>0$.  By
expansivity, if $d(x,y)<\alpha$, then there exists $j\geq 0$ such that $d(f^{j}(x), f^{j}(y)) > C$.  We are done by the previous lemma.
\end{proof}

\begin{lemma}\label{alfa}  If $\omega(z) = \emptyset$, then $\alpha (z)$ is either unbounded or a single periodic orbit.
 
\end{lemma}

\begin{proof}  Suppose that  $\alpha (z)$ is bounded, so that it is a compact invariant set $K$. We know that $f|_K$ is expansive:
there exists $C>0$ such that $x\neq y, x, y, \in K$ implies there exists $n\in \Z$ such that $d(f^n (x), f^n (y)) > C$. We claim that $K$ verifies the hypothesis of the previous lemma.

Take $\epsilon: \R^2\to \R$ as in Lemma \ref{epsilon}, and modify it if necessary such that $2\epsilon (x) < C$ for all $x\in K$.

Take $ \delta: \R^2\to \R$ as in the definition of topological shadowing, and for all $x\in K$, take $U_x= B(x, \delta(x)/2)$. Take $n_0$ such
that $f^{-n_0} (z)\in U=U_x$, for some  $x\in K$. By the choice of $\epsilon: \R^2\to \R$, the orbit of $z$ $C/2$-shadows
any pseudo-orbit such that $x_n = f^n (y), n<0$ for some $y\in U$, $x_n = f^{-n_0+n}(z), n \geq 0$.  This proves the claim, and therefore $K$ is finite.

Now, by Lemma \ref{una}
 $K$ must be a single 
periodic orbit. 
\end{proof}

\section{Time one maps}

 We recall the classical Poincar\'e-Bendixon's theorem on $\mathbb{S}^2$:
 
 \begin{teo}\label{pb}  Let $(f_t)_{t\in\R}$ be a flow defined by a $C^1 $-vector field on the sphere $\mathbb{S}^2$.  Then, the $\alpha$-limit and the $\omega$-limit of any orbit is either a singularity, a periodic orbit, 
 or a cycle of connections.
  
 \end{teo}
 
 Recall that a {\it connection} between two singularities $x_1$ and $x_2$ (not necesarilly different) is an orbit $x$ such that $\alpha (x) = x_1$ and $\omega (x) = x_2$ (or $\alpha (x) = 
 x_2$ and $\omega (x) = x_1$).\\
 
Throughout this section, we let $f:\R^2\to \R^2$ be a $\ta$ homeomorphism that is the time one map of a flow.  The orbit of a point $x$ for the flow will be noted $ {\mathcal O}(x)$.  Note
that the flow extends to the sphere $\mathbb{S}^2$ with a singularity at infinity.  We say that a connection between two singularities $x_1$ and $x_2$ is {\it finite} if $x_i\neq \infty$, $i=1,2$.

Our first goal is to prove:

\begin{teo}\label{teo1} $f$ is conjugate to a homothety.
 
\end{teo}

\begin{lemma}\label{exp}  There are no periodic orbits or finite connections.
 
\end{lemma}

\begin{proof}  We claim that any of those phenomena violate the topological expansivity.  Indeed, take a continuous and strictly positive function $\epsilon: \R^2\to \R$ such that for all 
$x, y \in \R^2, x\neq y$ there exists $k\in \Z$ 
  satisfying $||f^k (x)-f^k (y)|| > \epsilon (f^k (x))$.  Suppose that there is a finite connection.  Then, there exists $x, x_1, x_2\in \R^2$ such that $\alpha (x) = x_1$ and $\omega (x) = x_2$.
  Take $N$ large enough such that $|k|>N$ implies $f^k (x) \in B (x_1, m/2) \cup B (x_2, m/2) $, with $m= \min \{\epsilon (z): z \in \overline{{\mathcal O} (x)}\}$.  Note that
  if $y\in {\mathcal O}(x)$, enlarging $N$ if necessary we may assume that $|k|>N$ implies also  $f^k (y) \in B (x_1, m/2) \cup B (x_2, m/2) $. Now, take $\delta>0$
  such that $d(x,y)<\delta$ implies $d(f^n (x), f ^n (y))<m$ for all $|n|\leq N$.  Then, for all $k\in \Z$, $||f^k (x)-f^k (y)|| < m$, violating the expansivity condition.
  The proof for a periodic orbit is analogous and left to the reader. 
\end{proof}

The  previous lemma implies that if there is a cycle of connections containing $\infty$, there exists $x$ such that $\alpha(x) = \infty , \omega(x) = x_0$ with $x_0$ a singularity, 
and there exists $y$ such that $\alpha(y) = x_0 , \omega(y) = \infty$.

\begin{lemma}\label{sc}  There are no cycles of connections.
 
\end{lemma}

\begin{proof}  We have already seen that there are no finite connections (Lemma \ref{exp}).  We need to discard a cycle of connections containing $\infty$.  By the remark 
preceeding this lemma,  there exists $x$ such that $\alpha(x) = \infty , \omega(x) = x_0$ with $x_0$ a singularity, 
and there exists $y$ such that $\alpha(y) = x_0 , \omega(y) = \infty$.   Now, by Lemma 
\ref{epsilon}, one may choose $\epsilon: \R^ 2 \to \R$ a continuous positive map with the property that if $z\neq x$, then there exists $n\in \Z, n<0$ such that 
$||f^n (x)- f^n (z)|| > \epsilon (f^n (x))$ (the only orbit that $\epsilon$-past-shadows the orbit of $x$ is
the orbit of $x$ itself).  
 There exists $k>0$ such that 
$f^k (y)\notin B (x_0, \epsilon(x_0))$ because $\omega(y) = \infty$. Take $\delta: \R^ 2 \to \R$ a continuous positive map,
such that every $\delta$-pseudo-orbit is $\epsilon$-shadowed by an orbit.  We will finish the proof constructing a $\delta$-pseudo-orbit $(x_n)_{n\in \Z}$ that is not $\epsilon$-shadowed 
by any orbit. Let $n_0$ be such that
for all $n\geq n_0$ $f^n(x)\in B(x_0, \delta(x_0))$ and such that for all $n\leq n_0$ $f^n(y)\in B(x_0, \delta(x_0))$.  Define $x_n = f^n (x)$ for all $n\leq n_0$, $x_{n_0+1}=x_0$, $
x_{n_0+2}= f^{n_0} (y)$, $x_n = f(x_{n-1})$ 
for all $n> n_0+2$.  Note that $(x_n)_{n\in\Z}$ is a $\delta$-pseudo-orbit that is not $\epsilon$-shadowed, because the choice of $\epsilon$ implies that the only candidate is the orbit of $x$, 
but $f^k (y)\notin B (x_0, \epsilon(x_0))$.
\end{proof}

It follows by the Poincar\'e-Bendixon's theorem that both the $\alpha$ and $\omega$-limit of any point are either empty or consist of a single fixed point.  Moreover, if 
$\alpha (x) = \emptyset$, then $\omega(x) = x_0$, $x_0\in \fix(f)$ (and vice-versa).

\begin{lemma}  If there exists an  orbit $x$ such that $\omega(x) \neq \emptyset$, then $\omega(x)$ is Lyapunov stable.
 
\end{lemma}

\begin{proof}  By the previous remark, we may assume that $\omega(x) = x_0$ (and therefore $\alpha (x) = \emptyset$).  The result now  follows from Lemma \ref{stable}.
\end{proof}

\begin{lemma}  If there exists an  orbit $x$ such that $\omega(x) \neq \emptyset$, then $\omega(x)$ is a sink.
 
\end{lemma}

\begin{proof} Note that topological expansivity implies that $\fix (f)$ is a discrete set.  Take a neighborhood $U$ of $\omega(x)$ such that $U\cap \fix(f) = \omega(x)$.   By the previous 
lemma, there exists a neighborhood $V\subset U$ of $\omega(x)$ such that the 
$\omega$ limit of any orbit in $V$ cannot be empty, and therefore must be $\omega (x)$.  The result follows.
\end{proof}

\begin{rk} We have the analogous statement: If there exists an  orbit $x$ such that $\alpha(x) \neq \emptyset$, then $\alpha(x)$ is a source.
 
\end{rk}

\begin{lemma}  If there exists an  orbit $x$ such that $\omega(x) \neq \emptyset$, then $f$ is topologically conjugate to a homothety. 
 
\end{lemma}

\begin{proof}  By the previous lemma, there exists an open and invariant set $U$ such that $f|_U$ is conjugate to a homothety.  Let us show that $U= \R ^2$.  Otherwise, take 
$x\in \partial U$. Note that  the 
$\omega$ limit of $x$ must be empty, otherwise it would be a sink by the previous lemma, contradicting that it belongs to $\partial U$.  Also, the $\alpha$ limit of $x$ must be empty.
Otherwise, it would be a source and therefore the $\alpha$ limit of some point in $U$, contradicting that there are no connections.
 This contradicts Lemma \ref{inftyinfty}.
\end{proof}

\begin{rk}  Analogously,  if there exists an  orbit $x$ such that $\alpha(x) \neq \emptyset$, then $f$ is topologically conjugate to a homothety. 

\end{rk}
%
%
%
%

We are now ready to finish the proof of Theorem \ref{teo1}:

\begin{proof} Note that by Lemma \ref{inftyinfty} and the remark following Lemma \ref{sc}, there exists an  orbit $x$ such that $\omega(x) \neq \emptyset$, or there exists an 
orbit $x$ such that $\alpha(x) \neq \emptyset $.  The result now follows from the previous lemma and remark.
\end{proof}

\section{Attractor at infinity}

Throughout this section we assume that infinity is a topological attractor; that is, there exists an open simply connected proper subset $U$ with compact closure, such that 
$U \subset \inte(\overline {f(U)})$, and such that 
$ \cup_{n\geq 0} f^n (U)= \R ^2$.  We denote $K = \cap_{n\leq 0} f^n (U) $.  Note that $K$ is compact, invariant, connected and non-empty. 

\begin{lemma}\label{cuenca} There exists $\epsilon: \R^ 2 \to \R$ a continuous positive map  such that if $y\neq x$, $x,y\notin K$ there exists 
$ n>0$ such that 
$||f^n (x)- f^n (y)|| > \epsilon (f^n (x))$.  In particular,  if $(x_n)_{n\in \Z}$ is a pseudo-orbit such that $x_n = f^n (x)$ for all $n\geq n_0$, $x\notin K$, then the only possible orbit 
that $\epsilon$-shadows $(x_n)_n$ is that of $x$.
 
\end{lemma}

\begin{proof}  Just note that $f|_{\R^2 \backslash K}$
is conjugate to $x\mapsto \lambda x$, $\lambda >1$ on $\R^2 \backslash (0,0)$. 
\end{proof}

\begin{lemma} $K= \{x_0\}$
 
\end{lemma}

\begin{proof} We know that $f|_K$ is expansive:
there exists $C>0$ such that $x\neq y, x, y, \in K$ implies there exists $n\in \Z$ such that $d(f^n (x), f^n (y)) > C$. 
Take $\epsilon: \R^2\to \R$ as in the previous Lemma, and modify it if necessary such that $2\epsilon (x) < C$ for all $x\in K$.

Take $ \delta: \R^2\to \R$ as in the definition of topological shadowing, and for all $x\in K$ take $U= B(x, \delta(x)/2)$.  Note that the orbit of any $z\in U\backslash K$, 
$C/2$- shadows any pseudo-orbit such that $x_n = f^n (y), n<0$, $x_n = f^n (z), n\geq 0$ for some $y \in U$. So, Lemma \ref{autz} implies that  $K$ must be
finite
and as it is connected, a single point.
\end{proof}

As a corollary, we obtain:

\begin{teo}  If there is an attractor at infinity, then $f$ is conjugate to a homothety or  a reverse homothety.
 
\end{teo}

%
%

\section{There are no bounded basins}

Suppose there exists an open connected, simply connected  proper subset $U$ with compact closure such that 
$U \subset \inte(\overline {f(U)})$.  We have seen in the previous section that if $f$ is $\ta$ and
$ \cup_{n\geq 0} f^n (U)= \R ^2$, then $f$ is conjugate to homothety. We show in this section that $ \cup_{n\geq 0} f^n (U)$ must be unbounded.

\begin{lemma}  Let $D\subset \R ^2$ be an open topological disc with compact closure, and $f:\overline{D} \to \overline{D} $ an $\alpha$- expansive homeomorphism.  Then, there exists $x\in D$ such that $d(f^n (x), \partial D)\leq \alpha$ for all $n\in \Z$.
 
\end{lemma}

\begin{proof}  Consider the quotient space $X= \overline{D}/\partial D$.  Then, $X$ is homeomorphic to $\mathbb{S}^2$.  Moreover, we can define a metric on $X$ by letting 
$d([x], [y]) = d_{\R^2}(x, y)$ if $x, y \in D$, $d([x],[y]) = 0$ if $x, y \in \partial D$, $d([x],[y]) = d(x, \partial D)$ if $x\in D$, $y\in \partial D$.  Furthermore, 
$f:\overline{D} \to \overline{D} $ factors over $X$, and therefore cannot be expansive (by the classic Theorem in \cite{lew}). As $f:\overline D\to \overline D$ is $\alpha-$expansive, the result follows.
\end{proof}

\begin{teo}  Let $f: \R ^2\to \R^2$ be $\ta$. Suppose there exists an open simply connected  proper subset $U$ with compact closure such that 
$U \subset \inte(\overline {f(U)})$.  Then,  $ \cup_{n\geq 0} f^n (U)$ must be unbounded.
 
\end{teo}

\begin{proof}  Let $D = \cup_{n\geq 0} f^n (U)$ and suppose that it is bounded.  Then, $D$ is an open topological disc with compact closure, and $f:\overline{D} \to \overline{D} $ an
$\alpha$-expansive homeomorphism, with $\alpha = \min \{\epsilon (x): x \in \overline{D} \}$, where $\epsilon: \R^2 \to \R$ is given by topological expansivity.  Of course, it is also $\alpha'$-expansive for any $\alpha'<\alpha$. Take  $\alpha' < d(\partial D, K)$, where $K= \cap_{n\leq 0} f^n (U)$.

It follows from the previous lemma 
that there exists $x\in D$ such that $d(f^n (x), \partial D)\leq \alpha'$ for all $n\in \Z$.  By the choice of $\alpha', x\notin K$.  This is a contradiction, because if $n$ is large enough,
then $f^{-n}(x)$ lies outside the 
$\alpha'$-neighborhood of $\partial D$.
\end{proof}

%
%
%
%

\end{document}